\documentclass{amsart}
\usepackage[utf8]{inputenc}
\usepackage[OT1]{fontenc}
\usepackage{dsfont}
\usepackage{amsfonts}
\usepackage{amsmath}
\usepackage{wasysym}
\usepackage{amsthm}
\usepackage{amssymb}
\usepackage{amsthm}
\usepackage{float}
\usepackage{textcomp}
\usepackage{xcolor}
\usepackage{graphicx}
\usepackage{amssymb}
\usepackage{fancybox}
\usepackage{fancyhdr}
\usepackage{mathrsfs}
\usepackage{tikz}
\usepackage{centernot}
\usepackage{listings}
\usepackage{faktor}
\usepackage{bbm}
\usepackage{bm}
\usepackage{hyperref}
\usepackage{newlfont}
\usepackage{geometry}
\usepackage{ccaption}
\usepackage{esvect}
\usepackage{easyReview}
\usepackage[autostyle,italian=guillemets]{csquotes}
\usepackage[style=alphabetic,backend=biber]{biblatex} 
\usepackage{guit}
\theoremstyle{definition}
\newtheorem{Def}{Definition}[section]

\newtheorem{obs}[Def]{Remark}

\theoremstyle{theorem}
\newtheorem{prop}[Def]{Proposition}
\newtheorem{lema}[Def]{Lemma}

\newtheorem{corollario}[Def]{Corollary}
\newtheorem{teo}[Def]{Theorem}

\usepackage{color}
\usepackage[outerbars,color]{changebar}
\usepackage{ifthen}

\newboolean{usecolor}
\setboolean{usecolor}{true}

\ifthenelse{\boolean{usecolor}}
{
  \definecolor{colore}{cmyk}{0,1,0.6,0}
  \definecolor{coloregen}{cmyk}{0.7,0,1,0}
  \definecolor{coloresimo}{cmyk}{1,0.6,0,0}
  \definecolor{coloreconnor}{cmyk}{0,0.3,1,0}
}
{
  \definecolor{colore}{cmyk}{0,0,0,1}
  \definecolor{coloregen}{cmyk}{0,0,0,1}
  \definecolor{coloresimo}{cmyk}{0,0,0,1}
  \definecolor{coloreconnor}{cmyk}{0,0,0,1}
}

\usepackage[normalem]{ulem}

\title{The robustness of the generalized Gini index}
\author{M. Franciosi}
\author{S. Settepanella}
\author{A. Terni}
\address{Department of Mathematics, Pisa University, Pisa, Italy}
\address{Department of Mathematics, Hokkaido University, Japan}
\address{Department of Mathematics, Pisa University, Pisa, Italy}
\email{marco.franciosi@unipi.it}
\email{s.settepanella@math.sci.hokudai.ac.jp}
\email{terniale@gmail.com}

\thanks{}

\subjclass[2010]{ 28B05, 28A78 }
\keywords{Gini Index, zonoid,  empirical distribution, Hausdorff metric}

\begin{document}

\maketitle


\begin{abstract}
In  this paper we introduce a map $\Phi$, which we call \textit{zonoid map}, from  the space of all non-negative, finite Borel measures on $\mathbb{R}^n$ with finite first moment to the space of zonoids of $\mathbb{R}^n$. This map, connecting Borel measure theory with zonoids theory, allows us to  slightly generalize the  Gini volume introduced, in the contest of Industrial Economics, by  Dosi, Grazzi, Marengo and second author in 2016. This volume, based on the geometric notion of zonoid, is introduced as a measure of heterogeneity among firms in an industry and turned out to be quite interesting index as it is a multi-dimensional generalization of the well known and broadly used Gini index.\\
By exploiting the mathematical contest offered by our definition, we prove the continuity of the map $\Phi$ which, in turns, allows to prove the validity of a Glivenko-Cantelli theorem for our generalized Gini index and, hence, for the Gini volume. Both results, continuity of $\Phi$ and Glivenko-Cantelli theorem, are particularly useful when dealing with a huge amount of multi-dimensional data.

\end{abstract}


\section{Introduction}
Many problems in the social and system sciences are naturally multivariate and cannot be easily represented with a continuous or parametric approach.\\
An example is the economical production theory, that is, the theory that studies and represents the determinant factors driving production process dynamics. An industry is defined as a set of firms operating within the same sector and we can think about firm productivity as the ``ability" to turn inputs into outputs.\\
The classic approach in production theory is based on a number of assumptions regarding firms behaviour and firm production possibilities, in particular the profit maximization and cost minimization assumption. Following these assumptions, an \emph{ad hoc} parametrized family of production functions is introduced to assess firm productivity and efficiency and to estimate a number of economical indices. Such production functions satisfy, in addition, certain topological properties such as convexity and continuity, thus implying that firms with similar technologies will adopt analogous production techniques or, equivalently, firms tend to be \emph{homogeneous}.\\
Despite these assumptions, a growing availability of longitudinal microdata at firm-level has evidenced the fundamental role of heterogeneity in all relevant aspects regarding firms production activity, thus suggesting a switch from a continuous/parametric approach (which seems to be inadequate in presence of wide asymmetries) to a discrete/nonparametric point of view. Here geometry and geometric measure theory come into help.\\
To  evidence the fragilities of the classic theory, in 1981 Hildenbrand (cf. \cite{Hil}) adopted a different perspective,  by considering the empirical distribution induced by
  a set $X=\left\lbrace y_n\right\rbrace_{n=1,\dots,N}\subset\mathbb{R}^{m+1}_+$ of firms composing the industry (see section \ref{sec:main}  for details), and  introducing a geometric approach, 
 the \emph{zonoid representation}. Geometrically, a zonoid is a centrally symmetric, compact, convex set of the euclidean space which is induced by a Borel measure with finite expectation. In particular, the zonoid induced by the empirical distribution of a given industry is a convex polytope which is called a \emph{zonotope}. Zonotopes can also be written as a sum of line segments, in addition they are dense in the space of zonoids with respect to the topology induced by the Hausdorff metric.\\
More recently,  Dosi, Grazzi, Marengo and Settepanella  in \cite{DGMS} (see also \cite{DGLMS}) adopted Hildenbrand's construction to assess the rate of productivity and technological change of a given industry both on the microeconomic point of view (i.e. firm-level productivity) and on the macroeconomic point of view (i.e. aggregate productivity). Moreover a measure of heterogeneity of the industry, called the \emph{Gini volume}, is introduced. The above approach relies entirely on the geometry of the zonotope induced by the empirical distribution of the industry and it is highly nonparametric. On the other hand the Gini volume can also be seen as a measure of dispersion of the empirical distribution, indeed it is nothing else than a multi-dimensional generalization of the well known Gini index broadly used in social science and economics as measure of statistical dispersion  (see Remark \ref{rem:genG}). \\
The aim of this paper is to look at the Gini volume, i.e. the high-dimensional Gini index, in a slightly more general mathematical contest than the one in \cite{DGMS}.  This broader setting, which includes both, tools of measure theory and geometric properties of zonoids, allows firstly to generalize the definition of Gini volume to a broader class of measures, secondly to prove the validity of a law of large numbers type result for this generalized Gini index. The latter result turn out to be very useful when dealing with huge number of high dimensional data (often the case in applications).\\
Exploiting the dual aspect, provided by the zonoid representation, between the theory of Borel measures with finite first moment and the geometry of convex bodies, we introduce the zonoid map $\Phi\colon\mathcal{M}^n\to\mathcal{Z}^n$,  defined from the space $\mathcal{M}^n$ of all non-negative, finite Borel measures on $\mathbb{R}^n$ with finite first moment to the space $\mathcal{Z}^n$ of zonoids of $\mathbb{R}^n$.   Such map turns out to be continuous and allows us to prove the 
 validity of a Glivenko-Cantelli theorem for the Gini volume.
More precisely, we prove the continuity of $\Phi$ on the subspace of Borel probability measures with support on a compact $K \subset \mathbb{R}^n$ (see Proposition \ref{prop:cont}), which, jointly with a more general result which holds also for closed non compact cases, Theorem \ref{Gini extended}, provides the key ingredient to prove the main result of this paper, Theorem \ref{teo:main}. \\
Another interesting consequence of the continuity of $\Phi$ follows by Theorem \ref{density}, that is, every ``discrete" distribution $\mu$ can be substituted by a suitable ``continuous" distribution $\nu$ in such a way that the zonoid $Z(\nu)=\Phi(\nu)$ is a good approximation of $Z(\mu)=\Phi(\mu)$ at any desirable degree. This seems to suggest that a very large but finite dataset can be approximated with a continuous distribution, which may simplify much of the analysis without a great loss of informations. This will be object of further studies.\\
On the other hand, from the continuity of the map $\Phi$  we can deduce a notion of   robustness for the Gini volume. Indeed, if one consider 
 the empirical distribution induced by a concrete  dataset $X$ (e.g. of technological data),  subject to errors of various kind,  small changes in the values of the distribution lead  to a small change in the related zonoid, which  in turn implies that the Gini volume has a  small change as well. Moreover 
   the robustness of the Gini volume  implies that we can consider random samples  among the available data,  improving the computational aspect of the method.\\
To conclude it is worth to remark that our approach is in the same spirit of the one used in \cite{KM} by Koshevoy and Mosler. Their two generalizations of the Gini index to the multi-dimensional case are slightly different from our,  but  they both have many points in common with our generalization, for instance Corollary \ref{Glivenko zonoid}, and hence an analogous of Theorem \ref{teo:main}, applies to them too.\\
The paper is organized as follows. In Section \ref{sec:prelim} we introduce preliminary and basic notions needed in the rest of the paper. In Section \ref{sec:zonemp} we provide definition of empirical distributions and empirical zonoids proving that a Glivenko-Cantelli theorem for them holds.  
Finally in Section \ref{sec:main} we investigate the zonotope approach in production theory proposed in \cite{Hil} (1981), we generalize the Gini volume introduced in \cite{DGMS} (2016)  and we present a Glivenko-Cantelli result for this new generalized Gini index.

\section{Notation and preliminary results}\label{sec:prelim}
\noindent

A \emph{zonoid} is a convex body of $\mathbb{R}^n$ (i.e. it is compact and convex) which is centrally symmetric  and contains the origin. 
A \emph{zonotope} is a Minkowski sum of a finite number of line segments. In particular a   zonoid  is a polytope if and only if it is a zonotope. In this section we recall their relation with measure theory.
We mainly refer to \cite{Bol},  \cite{Bil}, and \cite{Mos}. For a more detailed presentation of the content of this and the following section in the contest of this paper see \cite{Ter}.

\subsection{An introduction to zonoids}

Let $\mathcal{M}^n$ be the set of all non-negative, finite Borel measures $\mu$ on $\mathbb{R}^n$ (with respect to the euclidean topology) whose first moment $$m(\mu)=\int_{\mathbb{R}^{n}}{x\ d\mu(x)} $$ is finite (here the integration is made component-wise). For every $\mu\in\mathcal{M}^n$, the \emph{zonoid} associated to the measure $\mu$ is the set $$Z(\mu)=\left\lbrace\int_{\mathbb{R}^n}{\phi(x)\cdot x\ d\mu(x)}\middle|\  \phi\colon\mathbb{R}^n\to\left[0,1\right] \textrm{ measurable }\right\rbrace\subseteq\mathbb{R}^n.$$
It can be considered as a geometric representation of the underlying measure: indeed, if we denote with $\mathcal{B}^n$ the class of Borel subsets of $\mathbb{R}^n$, then the zonoid $Z(\mu)$ can be seen as the convex hull of the image of the map $$F\colon\mathcal{B}^n\to\mathbb{R}^n\ ;\ F(B)=\int_{B}{x\ d\mu(x)}.$$ 
The zonoid $Z(\mu)$ is centrally symmetric about $\frac{1}{2}m(\mu)$ (sometimes we may also refer to $m(\mu)$ as the \emph{mean} or the \emph{gravity center} of the distribution).

On the functional point of view, if we denote by $\mathcal{Z}^n$ the set of zonoids of $\mathbb{R}^n$ we can consider the map $$\Phi\colon\mathcal{M}^n\to\mathcal{Z}^n\ ;\ \Phi(\mu)=Z(\mu),$$ which we call the \emph{zonoid map}. The zonoid map satisfies the following properties:
\begin{enumerate}
\item it is a homomorphism of semigroups: $Z(\mu+\nu)=Z(\mu)+Z(\nu)$ for every  $\mu$, $\nu\in\mathcal{M}^n$, where the sum on the right-hand side of the equality is the Minkowski sum;
\item it is positively homogeneous: for every $\alpha>0$ we have $Z(\alpha\mu)=\alpha Z(\mu)$;
\item it is linearly equivariant: for every linear map $L\colon\mathbb{R}^n\to\mathbb{R}^k$ we have $L(Z(\mu))=Z(L_*\mu)$, where $L_*\mu$ is the push-forward measure of $\mu$ with respect to $L$. In particular, the linear image of a zonoid is a zonoid.
\end{enumerate}
In addition, the zonoid map is clearly surjective but on the other hand it is not injective, since every zonoid is induced by a measure with support contained in the unitary sphere $S^{n-1}$ (for a proof, see \cite{Bol}).\\
Since we are mainly interested in probability measures, from now on we will focus our attention on the space of Borel probability measures $\mathcal{P}^n$ equipped with the topology induced by the weak convergence, which  is not a meaningful assumption  since the zonoid map $\Phi$ is positively homogeneous and any measure can be rescaled to a probability measure up to a normalization constant.

\subsection{Zonotopes and zonoids}\label{sub:Zon}
First of all note that  a zonoid is a zonotope if and only if it is induced by a finite atomic measure, i.e. a measure with finite support (cfr. \cite{Bol}).\\
Now, let  $\mathcal{K}^n$  be the set of convex bodies of $\mathbb{R}^n$. It is a classical result that if we equip $\mathcal{K}^n$ with the Hausdorff distance $$d_H(K,L)=\min\left\lbrace \epsilon\geq 0\vert\ K\subseteq L+\epsilon\cdot B^n,\ L\subseteq K+\epsilon\cdot B^n\right\rbrace,$$ where $B^n$ is the unit ball in $\mathbb{R}^n$, then $\left(\mathcal{K}^n, d_H\right)$ is a complete, sequentially compact metric space.\\
By its very definition we have the inclusion $\mathcal{Z}^n\subseteq\mathcal{K}^n$ and moreover, since the set of polytopes is dense in $\mathcal{K}^n$ with respect to the topology induced by the Hausdorff distance,  the subset of zonotopes is dense in $\mathcal{Z}^n$, that is, 
  every zonoid can be arbitrarily approximated (in the Hausdorff metric) by a zonotope, which has both a geometrical and combinatorial nature (see \cite{Bol} for the proof and some geometrical characterizations of a zonotope and \cite{Zie} for the combinatorial aspects). It is worth remarking that in combinatorial geometry there is an identification between zonotopes and arrangements of hyperplanes, although we won't deal with these aspects of the theory. Figure \ref{fig:zon} displays a zonotope generated by 4 line segments in $\mathbb{R}^3$.
\begin{center}
\begin{figure}
\includegraphics[scale=0.3]{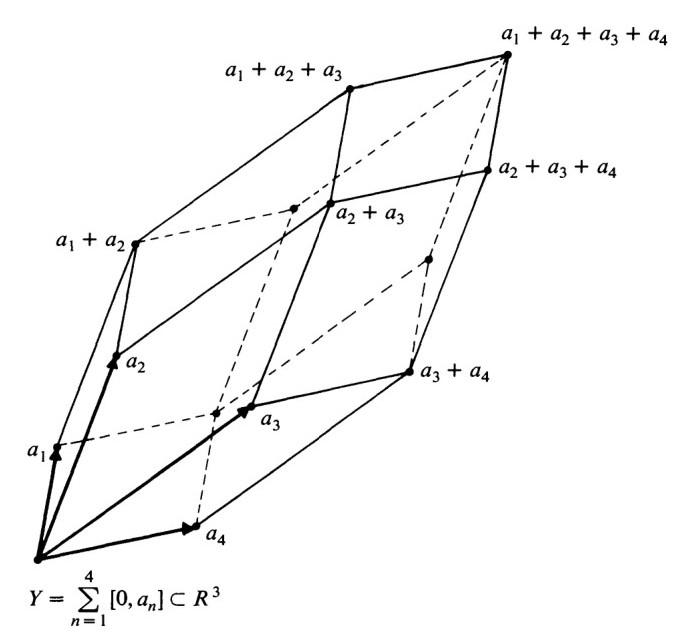}
\caption{Zonotope generated by 4 line segments.}\label{fig:zon}
\end{figure}
\end{center}
\subsection{Weak convergence of Borel distributions}\label{sub:weak}
In this section we will deal with the space of Borel probability measures, which can be equipped with the topology induced by the weak convergence. We also recall some classical facts which are valid in the general case of a complete separable metric space. Missing proofs and further details can be found in \cite{Bil}.\\
For a fixed $K$ non-empty, closed subset of $\mathbb{R}^n$ equipped with the subspace topology, we denote with $\mathcal{P}^n(K)$ the space of Borel probability measures with support contained in $K$. When $K=\mathbb{R}^n$, we simply write $\mathcal{P}^n$ for $\mathcal{P}^n(\mathbb{R}^n)$.\\
We recall that a sequence $\left(\mu_n\right)_{n\in\mathbb{N}}\subset\mathcal{P}^n(K)$ is said to converge \emph{weakly} to $\mu\in\mathcal{P}^n(K)$ if $$\lim_{n\to\infty}\int_{K}{f\ d\mu_n}=\int_{K}{f\ d\mu}$$ for every real-valued, continuous and bounded function $f$ defined on $K$. In this case we write $\mu_n\Rightarrow\mu$.\\
In our exposition, $K$ will be the whole space $\mathbb{R}^n$, a compact subset of it or the octant $\mathbb{R}_+^{n}=\left\lbrace x\in\mathbb{R}^{n}\middle|\ x\geq 0\right\rbrace$ (where the inequality $\geq$ is applied component by component). In this respect, if we denote with $\mathcal{C}^n$ the class of non-empty, compact subsets of $\mathbb{R}^n$,then the Riesz representation theorem for compact Hausdorff spaces implies that, given a compact subset $K\in\mathcal{C}^n$,  then $\mathcal{P}^n(K)$ can be seen as a subset of the space $(C^0(K;\mathbb{R}))^*$ (the  dual of  the space of continuous function $C^0(K;\mathbb{R})$) by means of the map $$\mu\mapsto\phi_{\mu}(\cdot)=\int_K{\cdot\ d\mu}.$$ As a consequence, the weak convergence in $\mathcal{P}^n(K)$ is induced by the weak-topology on $(C^0(K;\mathbb{R}))^*$, since every real-valued continuous function defined on a compact set $K$ is automatically bounded. More in general, without any assumptions of compactness on $K$, it is a classical result that the topology induced by the weak convergence is still metrizable, for instance, by the Prokhorv metric (for further details, see \cite{Bil}).\\
A fundamental example of Borel probability distribution on $\mathbb{R}^n$  is the \emph{Dirac measure}, that is, the probability measure that assigns unitary mass to a single point. Mathematically speaking, we write $\delta_x\in\mathcal{P}^n$ for the Dirac measure concentrated at the point $x\in\mathbb{R}^n$ and defined as follows: $$\delta_x(B)=\begin{cases} 0, & \mbox{if }x\notin B \\ 1, & \mbox{if }x\in B
\end{cases}$$ for every $B$ Borelian subset of $\mathbb{R}^n$.\\
Clearly, the support of the Dirac measure $\delta_x$ coincides with the singleton $\left\lbrace x\right\rbrace$. In addition, the space of convex combinations of Dirac measures $$\mathcal{Q}^n=\left\lbrace \sum_{i=1}^N{\alpha_i\delta_{x_i}}\in\mathcal{P}^n\colon\ N\in\mathbb{N},\ x_1,\dots,x_N\in\mathbb{R}^n,\ \sum_{i=1}^N{\alpha_i}=1,\ \alpha_i\in\left[0,1\right]\right\rbrace$$ 
coincides with the space of \emph{atomic probability measures} (i.e. those distributions with finite support) and by the separability of $\mathbb{R}^n$ the following theorem holds.
\begin{teo}\label{density}
The space of convex combinations of Dirac measures $\mathcal{Q}^n$ is a dense subset of $\mathcal{P}^n$ with respect to the topology induced by the weak convergence. In particular, the space of atomic probability measures is dense in $\mathcal{P}^n$.
\end{teo}
\noindent
The Dirac measure and Theorem \ref{density} play an important role in the next and in the last section of this paper.
\subsection{Continuity of the zonoid map} A family of measures $\left(\mu_i\right)_{i\in I}$ in $\mathcal{M}^n$ is \emph{uniformly integrable} if $$\lim_{\beta\to\infty}{\sup_{i\in I}{\int_{\parallel x\parallel\geq\beta}{\parallel x \parallel\ d\mu_i(x)}}}=0.$$\\
The following theorem, corollary of a more general result related to lift zonoids\footnote{For a more detailed discussion on lift zonoids in the contest of this work we refer the interested reader to \cite{Ter}.} (see Section 2.4 of \cite{Mos}), holds.
\begin{teo}\label{zonoid continuity}
Let $\left(\mu_k\right)_{k\in\mathbb{N}}$, $\mu\in\mathcal{M}^n$. If $\left(\mu_k\right)$ is uniformly integrable and $\mu_k\Rightarrow\mu$, then $Z(\mu_k)\xrightarrow{d_H}Z(\mu)$.
\end{teo}
\noindent
Observe in particular that a family of $\mathcal{M}^n$ is uniformly integrable when there exists a compact set $K$ of $\mathbb{R}^n$ which includes the support of all the measures of the family.\\
Now, let $\mathcal{P}^n_1(K):=\mathcal{P}^n(K)\cap\mathcal{M}^n$ be the space of probability measures with finite first moment and whose support is contained in a closed subset $K$ of $\mathbb{R}^{n}$. Note that  we have the equality $\mathcal{P}^n_1(K)=\mathcal{P}^n(K)$ when $K$ is compact. In particular,  a family of measures $\left(\mu_i\right)_{i\in I}$ in $\mathcal{P}^n(K)$ is always uniformly integrable when $K$ is compact.
Whence, as a corollary of Theorem \ref{zonoid continuity}, we have the following 
\begin{prop}[Countinuity on compact sets]\label{prop:cont}
For every $K\in\mathcal{C}^n$, the zonoid map $$\Phi\colon\mathcal{P}^n(K)\to\mathcal{Z}^n\ ;\ \Phi(\mu)=Z(\mu)$$ is continuous.
\end{prop} 
\begin{proof}
Every family of measures with support contained in a compact set is uniformly integrable. Hence, by Theorem \ref{zonoid continuity} the map $\Phi$ is a sequentially continuous map between two metric spaces, in particular it is a continuous map.
\end{proof}
\noindent
As aforementioned, beside the case in which $K$ is a compact set, it is of common interest the case in which $K$ coincides with $\mathbb{R}_+^{n}=\left\lbrace x\in\mathbb{R}^{n}\middle|\ x\geq 0\right\rbrace$.\\
Set $\mathcal{P}_1^+=\mathcal{P}_1^n(\mathbb{R}_+^{n})$. We are  interested in  describing  another sufficient condition, beside uniform integrability,  so that a family $\left(\mu_k\right)_{k\in\mathbb{N}}$ of measures in $\mathcal{P}_1^+$ needs to satisfy in order  to obtain a convergence result. 
\noindent
With this aim we recall  that  a sequence $\left(\mu_k\right)_{k\in\mathbb{N}}\subset\mathcal{P}_1=\mathcal{P}_1^n(\mathbb{R}^{n})$ is said to be convergent \emph{in mean} to $\mu\in\mathcal{P}_1$ (write $\mu_k\xrightarrow{\mathcal{M}}\mu$) if it converges weakly to $\mu$ and the sequence $\left( m(\mu_k)\right)$ converges to $m(\mu)$ for $k\to\infty$.
\begin{teo}\label{zonoid mean continuity}
Given $\left(\mu_k\right)_{k\in\mathbb{N}}\subset\mathcal{P}_1^+$ and $\mu\in\mathcal{P}_1^+$, then $\mu_k\xrightarrow{\mathcal{M}}\mu$ implies $Z(\mu_k)\xrightarrow{d_H}Z(\mu)$.
\end{teo}
\begin{proof}
See \cite{Hil}
\end{proof}
Remark that for any $K$ compact subset of $\mathbb{R}^{n}$, a sequence $\left(\mu_k\right)_{k\in\mathbb{N}}\subset\mathcal{P}(K)$ is convergent in mean to $\mu\in\mathcal{P}(K)$ if and only if it is weakly convergent to $\mu$. 

\section{Zonoids related to empirical distributions}\label{sec:zonemp}
 We begin with a definition:
\begin{Def}
Let $X=\left\lbrace y_k\right\rbrace_{k=1,\dots,N}\subset\mathbb{R}^{n}$ be a finite set. The \emph{empirical distribution} of $X$ is the Borel measure $$\widehat{\mu}=\frac{1}{N}\sum_{k=1}^N{\delta_{y_k}},$$ the zonoid related to the empirical distribution $Z\left(\widehat{\mu}\right)$ is the \emph{empirical zonoid}.
\end{Def}
As noticed in Subsection \ref{sub:Zon}, since $\widehat{\mu}$ is a measure with finite support then the induced  empirical zonoid $Z(\widehat{\mu})$ is indeed a zonotope.\\
In many application contexts, the empirical distribution is induced by a dataset $X$ of technological data which are subject to errors of various kind. Hence, it is desirable that a small change in the distribution should lead only to a small change in the related zonoid or, equivalently, that the map $\Phi$ should satisfy a continuity result. This is quite useful when one needs to rely on samples, for instance when the collection of technological data (e.g. the production activity of an industry in several countries) is time consuming and costly. 
In this respect,  in Proposition \ref{prop:cont} we have already stated a continuity result for zonoids in the compact case. Analogous result can be stated for the non compact case $\mathcal{P}_1^+$.
The following version of Glivenko-Cantelli Theorem for separable metric spaces, whose proof can be found in \cite{Var}, holds.
\begin{teo}[Glivenko-Cantelli]\label{Glivenko-Cantelli}
Let $\left(E,d\right)$ be a separable metric space and $X_1,X_2,\dots$ be independent $E$-valued random variables with distribution $\mu$ (we consider on $E$ the $\sigma$-field of Borelian subsets). Let $\widehat{\mu}_N$ be the empirical measure $$\widehat{\mu}_N=\frac{1}{N}{\sum_{i=1}^N{\delta_{X_i}}};$$ then we have $\widehat{\mu}_N\Rightarrow\mu$ for $N\to\infty$ with probability 1.
\end{teo}
\noindent
Notice that Theorem \ref{Glivenko-Cantelli} implies that the empirical zonoid which is derived from a large sample of the true distribution $\mu$ will yield a good approximation of $Z(\mu)$. A consequence of Theorem  \ref{Glivenko-Cantelli} and Theorem \ref{zonoid mean continuity} is the following corollary.
\begin{corollario}\label{Glivenko zonoid}
Let $X_1,X_2,\dots$ be independent $\mathbb{R}^{n}_+$-valued random variables with distribution $\mu\in\mathcal{P}_1^+$. Let $\widehat{\mu}_N$ be the empirical measure $$\widehat{\mu}_N=\frac{1}{N}{\sum_{i=1}^N{\delta_{X_i}}};$$ then we have $$Z(\widehat{\mu}_N)\xrightarrow{d_H}Z(\mu)$$ with probability 1.
\end{corollario}
\begin{proof}
The usual law of large numbers implies $m(\widehat{\mu}_N)\xrightarrow[]{\parallel\cdot\parallel}m(\mu)$ with probability 1, hence we can combine it with Theorem \ref{Glivenko-Cantelli} to conclude that $\widehat{\mu}_N\xrightarrow{\mathcal{M}}\mu$ with probability 1 and thus the thesis follows by Theorem \ref{zonoid mean continuity}.
\end{proof}
\noindent
To conclude we remark 
 that Corollary \ref{Glivenko zonoid} can actually be extended to $X_1,X_2,\dots$ independent $\mathbb{R}^{n}$-valued random variables with distribution $\mu\in\mathcal{P}_1$ (for a proof, see \cite{Mos}).

\section{Applications to Production Theory: the generalized  Gini index}\label{sec:main}
\noindent
In recent years, a wide literature based upon empirical analyses has robustly evidenced the permeating presence of heterogeneity in all relevant aspects of the dynamics of  production processes.  
Recently, Dosi, Grazzi, Marengo and Settepanella  (see  \cite{DGMS}),  introduced the  \emph{Gini Volume}, a new non parametric index to assess the degree of heterogeneity of an industry. Their construction is based on the paper   \cite{Hil} by Hildebrand, in which the author  applies the theory of zonoids to the one of industrial production.\\
In this section we recall the definition of such  index, we provide a slight generalization by  means of the zonoid representation and we prove the validity of a Glivenko-Cantelli type result.

\subsection{The zonotope approach} 
In 1981, Hildenbrand suggested a geometrical representation of a given industry. Such representation is highly nonparametric and it is based upon observed production activity, that is, every industry is represented as a set $$X=\left\lbrace y_n\right\rbrace_{n=1,\dots,N}\subset\mathbb{R}^{m+1}_+,$$ where:
\begin{itemize}
\item $N$ is the number of productive units (i.e. the firms) making up the industry;
\item every point $y_n$ is called the \emph{observed} production activity of the $n$-th firm;
\item the first $m$ coordinates of $y_n$ represent the input quantities adopted by the $n$-th firm and the last coordinate is the output quantity produced under the period of observation (we say we are in the $m$-input, 1-output case)\footnote{We slightly changed notation with respect to the previous sections replacing $\mathbb{R}^n$ with $\mathbb{R}^{m+1}$ to be consistent with notation in \cite{Hil} and \cite{DGMS}.}.
\end{itemize}
Let $X=\left\lbrace y_n\right\rbrace_{n=1,\dots,N}\subset\mathbb{R}^{m+1}_+$ be a fixed set which represents a given industry.  In \cite{Hil} Hildebrand defines the \emph{production set} of the $n$-th firm as the line segment $$\left[0,y_n\right].$$ The \emph{size} of the $n$-th firm is the euclidean norm of the vector $\vv{0y_n},$  $\parallel y_n\parallel$.\\
Notice that the definition of production set corresponds, roughly speaking, to the assumption that each firm doesn't change its production activity under the period of observation, thus it can be seen as a first order approximation of the problem. In \cite{Hil} there is a geometric representation of the industry $X$ from the aggregate point of view.
\begin{Def} 
The \emph{short-run total production set} of the industry $X$ is the Minkowski sum of the production set of each firm, that is, the zonotope $$Z=\sum_{n=1}^N{\left[0,y_n\right]}.$$
\end{Def}
Consider the empirical measure of the industry $X$, that is, the measure $$\widehat{\mu}=\frac{1}{N}\sum_{n=1}^N{\delta_{y_n}}.$$ We recall that $\widehat{\mu}$ is a probability measure with finite support, hence it is an atomic probability with finite mean and we have $\widehat{\mu}\in\mathcal{P}_1^+$. As noted by Hildenbrand, for every Borelian set $B$ the quantity $100\cdot \widehat{\mu}(B)$ can be seen as the percentage of production units having their characteristics in the set $B$.
\begin{Def}
The \emph{short-run mean production set} of the industry $X$ is the zonoid $Z(\widehat{\mu})$, where $\widehat{\mu}$ is the empirical distribution of $X$.
\end{Def}
The term ``mean" adopted in the above definition follows from the observation that $Z(\widehat{\mu})$ is a homothetic copy of the short-run total production set $Z$, indeed we have $$Z=N\cdot Z(\widehat{\mu}).$$
\begin{obs}
As a convex body, every zonoid $Z(\mu)$ is uniquely determined by its \emph{support function}, defined as follows: $$\psi_{\mu}\colon\mathbb{R}^n\to\mathbb{R}\ ;\ \psi_{\mu}(\xi)=\sup\left\lbrace \left\langle x,\xi\right\rangle\middle|\ x\in Z(\widehat{\mu})\right\rbrace.$$ It is an interesting fact that in \cite{Hil}, an economic interpretation of the support function of $Z(\widehat{\mu})$ is given: if we write $\xi=\left( -\xi_1,\dots,-\xi_m,\xi_{m+1}\right)\in\mathbb{R}^{m+1}$, then the quantity $\psi_{\widehat{\mu}}(\xi)=\sup\left\lbrace \left\langle x,\xi\right\rangle\middle|\ x\in Z(\widehat{\mu})\right\rbrace\ $ can be considered as the maximum mean profit with respect to the price system $\xi$ subject to the technological restrictions defined by the mean production set $Z(\widehat{\mu})$.
\end{obs}
Building by Hildenbrand's work,  Dosi, Grazzi, Marengo and Settepanella in \cite{DGMS} introduce a new framework to assess firm level heterogeneity and to study the rate and direction of technical change, which we are now going to examine.
\subsection{Heterogeneity and Gini volume}
Empirical evidence reports a wide and persistent heterogeneity across firms operating in the same industry, thus the phenomenon requires attention.\\
Intuitively, heterogeneity can be associated in mathematical statistics to the variance, namely it measures how much the industry is far from being homogeneous or, equivalently, how much the various productive units differ from the ``mean" productive unit. 
\begin{Def} Let $X=\left\lbrace y_n\right\rbrace_{n=1,\dots,N}\subset\mathbb{R}^{m+1}_+$  be an industry and let  $Z$
be the related short-run total production set. The \emph{total production activity} is the sum $$\Sigma_Z=\sum_{n=1}^{N}{y_n} \in Z.$$
\end{Def}
Geometrically, the line segment $d_Z:=\left[ 0, \Sigma_Z\right]$ is the main diagonal of the zonotope $Z$ and it seems to be a good candidate to represent the ``mean" productive technology of the industry: indeed we have $$\frac{\Sigma_Z}{N}=m\left(\widehat{\mu}\right),$$ where $m\left(\widehat{\mu}\right)$ is the expectation of the empirical measure $\widehat{\mu}$ related to the industry (i.e. the set) $X$.\\
For a better visualization, let us analyse two limit cases, one the opposite of the other:
\begin{itemize}
\item \textbf{Maximal homogeneity:} every production set lies on the line spanned by the main diagonal $d_Z$. This corresponds to the situation where every production activity adopts the same productive technology and any two of them only differ by their intensities (i.e. their size). In this case, we have $Z=d_Z$, which is a zonotope with null volume;
\item \textbf{Maximal heterogeneity:} production sets are represented by segments on positive semi-axis and the zonotope $Z$ is a parallelotope in $\mathbb{R}^{m+1}$ with diagonal $d_Z$. This case has to be regarded as a limit case: indeed, production sets on positive semi-axis would imply that there are firms with either nonzero inputs and zero output or nonzero output and zero inputs, which is quite absurd. 
\end{itemize}
Building from these two cases,  Dosi et alt. in \cite{DGMS} defines the following index as a candidate measure of heterogeneity:
\begin{Def}
The \emph{Gini volume} for the short run total production set $Z$ induced by the industry $X$ is the ratio $$G(Z)=\frac{V_{m+1}(Z)}{V_{m+1}(P_Z)}\in\mathbb{R},$$ where $P_Z$ is the $(m+1)$-dimensional parallelotope $$P_Z:=\left\lbrace z\in\mathbb{R}^{m+1}\colon\ 0\leq z\leq \sum_{i=1}^N{y_n}=\Sigma_Z\right\rbrace.$$
\end{Def}
Observe that the Gini volume does not depend on the units of measure or the number of firms, thus it allows comparisons across space and time. In addition, we have the inequality $$0\leq G(Z)\leq 1,$$ where the minimum is attained at the maximal homogeneity case and the maximum is attained in the maximal heterogeneity case.
\begin{obs}
Clearly, the inequality $N\geq m+1$ must be satisfied, otherwise the Gini volume would be null (observe that in applications the number $N$ is usually large). When $N\geq m+1$, then we have the equality $$V_{m+1}(Z)=\sum_{i\in I}{\vert \Delta_{i}\vert},$$
where $I=\left\lbrace i=(i_1,\dots ,i_{m+1})\in\mathbb{R}^{m+1}\ \vert\ 1\leq i_1<\dots <i_{m+1}\leq N\right\rbrace$ and $\Delta_i$ is the determinant of the matrix whose rows are the vectors $\left\lbrace y_{i_1},\dots ,y_{i_{m+1}}\right\rbrace$.
On the other hand, we have $$V_{m+1}(P_Z)={\Pi_{i=1}^{m+1}{\left\langle\Sigma_Z, e_i\right\rangle}},$$ where $\left\lbrace e_i\right\rbrace_{i=1,\dots,m+1}$ is the canonical basis and $\left\langle\ ,\ \right\rangle$ is the standard scalar product.
\end{obs}
We provide the following continuity result on the Gini volume.  
\begin{teo}\label{Gini continuity}
Let $\mathcal{Z}^{m+1}_+$ be the space of zonotopes $Z$ that  are contained in $\mathbb{R}^{m+1}_+$ and  verify ${V_{m+1}(P_Z)}\neq 0$.
Then the Gini volume, seen as a real-valued function defined on $\mathcal{Z}^{m+1}_+$ equipped with the topology induced by the Hausdorff metric,  is continuous.
\end{teo}
\noindent
In order to prove this theorem we need the following lemma.
\begin{lema}\label{volume functional}
The volume functional $V_{m+1}$ is continuous on the space of convex bodies in $\mathbb{R}^{m+1}$ with respect to the Hausdorff metric.
\end{lema}
A proof can be found in \cite{Sch}.
\begin{proof}[Proof of Theorem \ref{Gini continuity}]
Since the volume functional is continuous by Lemma \ref{volume functional}, the only thing left to prove is the continuity of the map $$Z\mapsto P_Z.$$
Indeed, the function is also uniformly continuous, in fact for every couple of zonotopes $Z$, $Z'$ with $d_H(Z,Z')\leq\epsilon$ we have $$Z\subseteq Z'+\epsilon\cdot B^{m+1}\subseteq P_{Z'}+\epsilon\cdot B^{m+1},$$ hence the inclusion $$P_Z\subseteq P_{Z'}+\epsilon\cdot B^{m+1}$$ follows easily from the definition of $P_Z$. Clearly we can exchange the roles of $Z$ and $Z'$ to get the inequality $$d_H(P_Z,P_{Z'})\leq\epsilon.$$
\end{proof}
The above defined Gini volume can be expressed even in terms of the empirical distribution $\widehat{\mu}$ of the set $X$:
\begin{obs}\label{Gini index and empirical measure}
Note that, for every $\mu\in\mathcal{P}_1^+$, the associated zonoid $Z(\mu)$ is contained in the $m+1$-dimensional parallelotope 
\begin{equation}
P(\mu):=\left\lbrace z\in\mathbb{R}^{m+1}\colon\ 0\leq z\leq m(\mu)\right\rbrace,
\end{equation}
 where $\leq$ is applied component by component. In this respect we have the equality $$G(Z)=\frac{V_{m+1}(Z(\widehat{\mu}))}{V_{m+1}(P(\widehat{\mu}))}=G\left(Z\left(\widehat{\mu}\right)\right),$$ which can be easily deduced from the relations $Z=N\cdot Z(\widehat{\mu})$ and $P_Z=N\cdot P(\widehat{\mu})$. In particular, we have ${V_{m+1}(P_Z)}\neq 0$ if and only if the expectation $m(\widehat{\mu})\in\mathbb{R}^{m+1}_+$ is a vector with strictly positive coordinates.
\end{obs}
\subsection{The generalized Gini index}
Remark \ref{Gini index and empirical measure} suggests an extension of the Gini volume definition to the set of zonoids induced by $\mathcal{P}_1^+$:
\begin{Def}\label{def:gini}
Let $\mu\in\mathcal{P}_1^+$ be a Borel distribution such that $m(\mu)$ is a vector with strictly positive coordinates. The {\em Gini index} related to $\mu$ is the ratio $$G\left(Z\left(\mu\right)\right)=\frac{V_{m+1}(Z(\mu))}{V_{m+1}(P(\mu))},$$ where $P(\mu)$ is the parallelotope defined in Remark \ref{Gini index and empirical measure}.
\end{Def}
\begin{obs}\label{rem:genG}
Let $\mu\in\mathcal{P}_1^1$ be a univariate probability distribution with support contained in $\mathbb{R}_+$  and such that $m(\mu)\neq 0$ (equivalently $m(\mu)>0$). Consider the \emph{lifted measure} induced by $\mu$, that is, the bivariate probability distribution $$\overline{\mu}=\delta_1\otimes\mu,$$ where $\delta_1\in\mathcal{P}_1^1$ is the Dirac measure which assigns unitary mass to the point 1. Observe that we can write $\overline{\mu}\in\mathcal{P}_1^+$ if we set $m+1=2$.\\
In \cite{Mos} it is proved that the zonoid $Z(\overline{\mu})$ (which is also called the \emph{lift zonoid} induced by $\mu$) is a bidimensional convex body bordered by two curves, the \emph{generalized Lorenz curve} and the \emph{dual generalized Lorenz curve} induced by $\mu$.
We recall that the generalized Lorenz curve induced by the distribution $\mu$ is  defined as $$L_{\mu}(t)=\left(t,\int_0^t{Q_\mu(s)\ ds}\right)\ ,\ 0\leq t\leq 1,$$ where $Q_{\mu}(s)$ is the quantile function of $\mu$: $$Q_{\mu}(s)=\inf\left\lbrace x\in\mathbb{R}\colon\ \mu\left(\left]-\infty,x\right]\right)\geq s\right\rbrace,$$ whereas the dual generalized Lorenz curve is obtained by symmetrization of the generalized Lorenz curve with respect to the center of symmetry of $Z(\overline{\mu})$, that is, the point $C=\left(\frac{1}{2},\frac{1}{2}m(\mu)\right)\in\mathbb{R}^2$.
Figure \ref{fig:gini} shows the zonoid $Z(\overline{\mu})$ and the parallelotope $P(\overline{\mu})$ when $\mu$ is the exponential distribution with parameter 1, that is, when $\mu=Exp(1)$.
\begin{center}
\begin{figure}
\includegraphics[scale=0.19]{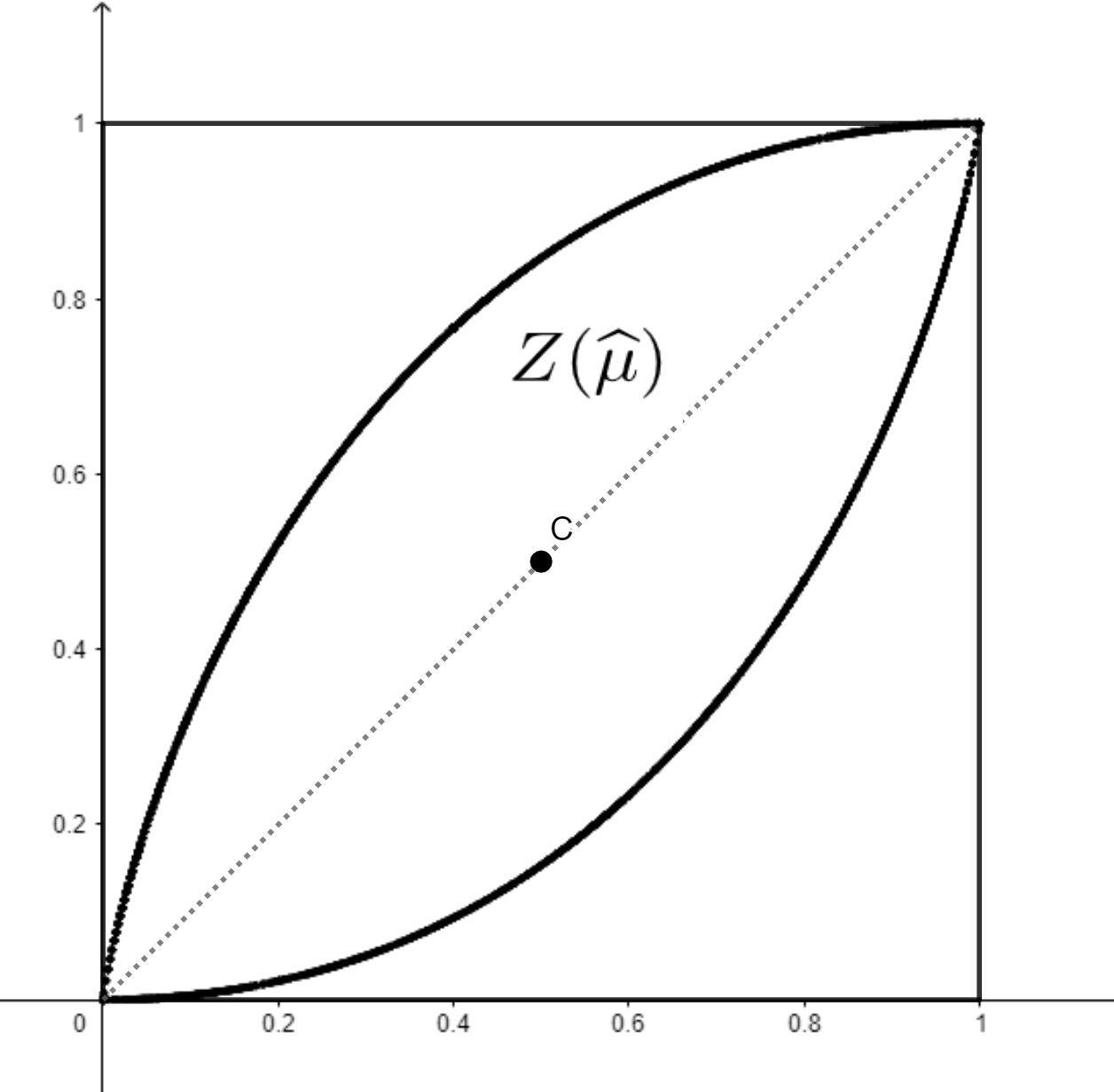}
\includegraphics[scale=0.175]{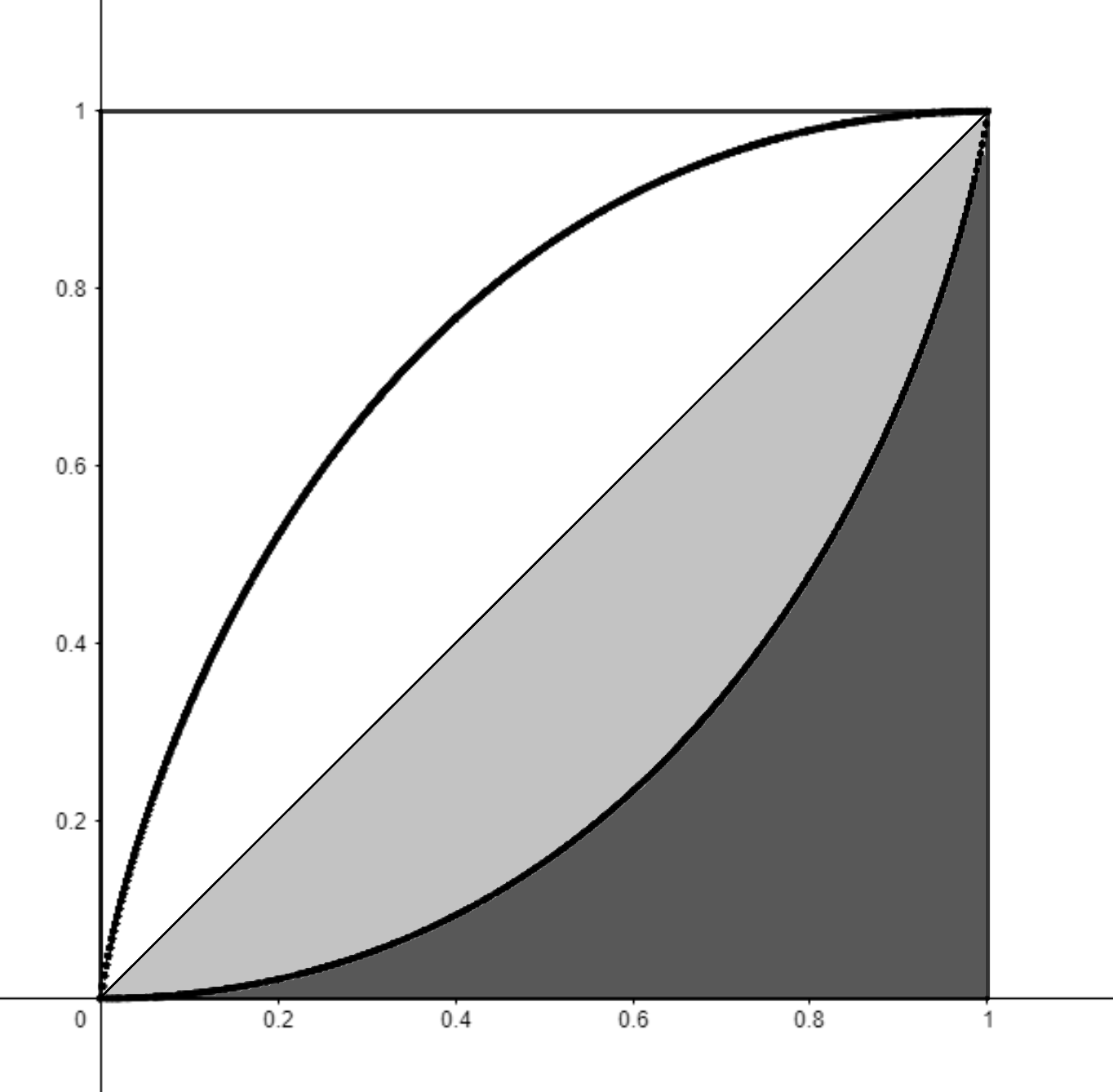}
\caption{Lorenz curve.}\label{fig:gini}
\end{figure}
\end{center}
The generalized Lorenz curve is represented by the lower curve below the dotted line displayed in the figure (which corresponds to the segment whose endpoints are the origin and the point $\left(1,m(\mu)\right)$), whereas the dual generalized Lorenz curve is represented by the upper curve above the dotted line. On the other hand, the rectangle (the square) containing the zonoid in Figure \ref{fig:gini} coincides with the 2-dimensional parallelotope $P(\overline{\mu})$.\\
On the right, the light grey surface represents the portion of plane between the dotted line and the generalized Lorenz curve, whereas the dark grey surface represents the portion of $P(\overline{\mu})$ which is situated below the generalized Lorenz curve. By a symmetry argument, we can observe that the proposed generalization in Definition \ref{def:gini} graphically coincides with the ratio between the area of the light grey surface and the area of the dark grey surface united with the light grey surface, hence the term \emph{generalized} Gini index referred to Definition \ref{def:gini} is justified.
\end{obs}
\noindent
By means of results showed in section 2 and 3 applied to the generalized Gini index in Definition \ref{def:gini} we obtain the following  continuity result that, in particular, applies to the index of heterogeneity proposed in \cite{DGMS}. 
\begin{teo}\label{Gini extended}
Let $\left(\mu_k\right)_{k\in\mathbb{N}}\subset\mathcal{P}_1^+$, $\mu\in\mathcal{P}_1^+$ be Borel distributions such that $V_{m+1}(P(\mu))\neq 0$ and $V_{m+1}(P(\mu_k))\neq 0$ for every index $k$, where $P(\mu)$ is the parallelotope defined in Remark \ref{Gini index and empirical measure}. If $\mu_k\xrightarrow{\mathcal{M}}\mu,$ then the sequence $G(Z(\mu_k))$ is convergent to $G(Z(\mu))$.
\end{teo}
\begin{proof}
The proof follows immediately by Theorem \ref{zonoid mean continuity} and the observation that if $\mu_k\xrightarrow{\mathcal{M}}\mu,$ then $P(\mu_k)\xrightarrow{d_H}P(\mu)$.
\end{proof}
\noindent
Finally we present a Glivenko-Cantelli type result, which may be used in a more general contest, beside the production theory one.
\begin{teo}\label{teo:main}
Let $\mu\in\mathcal{P}_1^+$ be a Borel distribution such that the expectation $m(\mu)$ is a vector with strictly positive coordinates and let $X_1,X_2,\dots$ be independent $\mathbb{R}^{m+1}_+$-valued random variables with distribution $\mu$. Let $\widehat{\mu}_N$ be the empirical measure $$\widehat{\mu}_N=\frac{1}{N}{\sum_{i=1}^N{\delta_{X_i}}};$$ then the sequence $G(Z(\widehat{\mu}_N))$ is eventually defined and it is convergent to $G(Z(\mu))$ with probability 1.
\end{teo}
\begin{proof}
Observe that, since we have $\mu\in\mathcal{P}_1^+$, the expectation $m(\mu)$ is a vector with strictly positive coordinates if and only if the parallelotope $P(\mu)$ has non-empty interior or, equivalently, if and only if $V_{m+1}\left(P(\mu)\right)\neq 0$. By the usual law of large numbers we have $m(\widehat{\mu}_N)\rightarrow m(\mu)$ with probability 1, hence the sequence of parallelotopes $P(\widehat{\mu}_N)$ has eventually non-empty interior and thus the index $G(Z(\widehat{\mu}_N))$ is eventually well defined almost surely. At this point, we can conclude by Theorem \ref{Glivenko-Cantelli} and Theorem \ref{Gini extended}.
\end{proof}

\end{document}